\documentclass[12pt]{amsart}

\usepackage{amssymb,color}
\usepackage{amsfonts}
\usepackage{amsmath}
\usepackage{euscript}
\usepackage{enumerate}
\usepackage{graphics}
\usepackage{color}
\usepackage{pdfsync}
\newtheorem{theorem}{Theorem}[section]
\newtheorem{lemma}[theorem]{Lemma}
\newtheorem{note}[theorem]{Note}
\newtheorem{prop}[theorem]{Proposition}

\newtheorem*{Theorem1'}{Theorem 1'}

\theoremstyle{definition}
\newtheorem{definition}[theorem]{Definition}

\theoremstyle{remark}

\numberwithin{equation}{section}

\newcommand \C{{\mathbb C}}
\renewcommand \O{{\mathcal O}}
\newcommand \R{{\mathcal R}}
\renewcommand \ker{{\mathrm {ker}}}

\newcommand \Hom{{\mathrm {Hom}}}

\newcommand \Irr{{\mathrm {Irr}}}
\newcommand \Top{{{Top}}}
\newcommand \Bot{{{Bot}}}

\newcommand \GL{{\mathrm {GL}}}
\newcommand \Sp{{\mathrm {Sp}}}
\newcommand \U{{\mathrm {U}}}
\newcommand \SU{{\mathrm {SU}}}
\newcommand \SL{{\mathrm {SL}}}

\newcommand \m{{\mathfrak {m}}}
\newcommand \n{{\mathfrak {n}}}
\renewcommand \i{{\mathfrak {i}}}
\renewcommand \j{{\mathfrak {j}}}
\renewcommand \r{{\mathfrak {r}}}

\begin{document}

\title[Weil representations of unitary groups]{Weil representations of unitary groups over ramified extensions of finite local rings with odd nilpotency length}

\author[A. Herman]{Allen Herman$^*$}
\address{Department of Mathematics and Statistics, University of Regina, Regina, Canada, S4S 0A2}
\email{Allen.Herman@uregina.ca}
\author[M. Shau]{Moumita Shau}
\address{Department of Mathematics and Statistics, University of Regina, Regina, Canada, S4S 0A2}
\email{tuktukishau@gmail.com}
\author[F. Szechtman]{Fernando Szechtman$^*$}
\address{Department of Mathematics and Statistics, University of Regina, Regina, Canada, S4S 0A2}
\email{fernando.szechtman@gmail.com}

\thanks{$^*$These authors were supported by NSERC Discovery Grants.}
\date{}



\begin{abstract}
We find the irreducible decomposition of the Weil representation
of the unitary group $\U_{2n}(A)$, where $A$ is a ramified quadratic extension
of a finite, commutative, local, principal ideal ring $R$ and the nilpotency degree of
the maximal ideal of $A$ is odd. We show in particular that this
Weil representation is multiplicity free. Restriction to the
special unitary group $\SU_{2n}(A)$ preserves irreducibility and
multiplicity freeness provided $n>1$.
\end{abstract}

\maketitle

\section{Introduction}

Weil representations of unitary groups over finite fields were
first considered by G\'erardin \cite{G}, who gave an explicit formula
for the Weil character as well as its decomposition into
irreducible constituents. Since then, Weil representations of
unitary and special unitary groups over finite fields have
attracted considerable attention. For instance, Tiep \cite{T} studies when the Weil
constituents lead to globally irreducible representations; he also finds some Schur indices
and provides a branching formula for the Weil character; Tiep
and Zalesskii \cite{TZ} show how to determine the Weil components of
$\SU_n(q)$ by their restrictions to standard subgroups; Hiss and
Malle \cite{HM} as well as Guralnick,  Magaard, Saxl and Tiep
\cite{GMST} show that the Weil
constituents and the trivial module are the only irreducible
modules of $\SU_n(q)$ having ``small" degree; Hiss and Zalesskii
\cite{HZ} decompose the tensor product of the Weil and Steinberg
characters of $\U_n(q)$ and prove this decomposition to be
multiplicity free; Pallikaros and Zalesskii \cite{PZ} show that
the restriction of the Weil representation of $\U_n(q)$ to certain
subgroups is multiplicity free, while Pallikaros and Ward
\cite{PW} find necessary conditions for other types of subgroups
to have this same property.

Weil representations of unitary groups over an unramified
extension of a finite, commutative, local, principal ideal ring were considered by
Gow and Szechtman \cite{GS}, who gave a formula for the Weil
character which does not depend on G\'erardin's formula for the field
case. The decomposition of the Weil character into irreducible
constituents was later obtained in~\cite{S}.

Weil representations of unitary groups over certain noncommutative rings were constructed by
Guti\'errez, Pantoja and Soto-Andrade \cite{GPS}, by first giving
a presentation for such groups and then assigning suitable linear
operators to the generators in such a way that the defining
relations were preserved.

In this article we deal with Weil representations of unitary groups over suitable rings, as follows.
Let $\O$ be a local PID that is not a field having maximal ideal $\O s$ and finite residue field $F_q$ of odd characteristic,
and let $\R=\O[X]/(X^2-s)$ be a ramified quadratic extension of~$\O$. We have  $\R=\O\oplus \O t$, where $1,t$ is an $\O$-basis of $\R$ and $t^2=s$. The ring $\R$ is also a local PID with maximal ideal $\R t=\O s\oplus \O t$ and residue field $F_q$. Moreover, $\R$
has an involution $\sigma$ that fixes every element of $\O$ and sends $t$ to $-t$.

For $m\geq 1$, we set $A=\R/\R t^{m}$ and $R=\O/\O s^{\lceil m/2\rceil}$. Note that $A$ (resp. $R$) is
a finite, commutative, local, principal ideal ring with maximal ideal $\r=A\pi$ (resp. $\m=Rp$), where $\pi$ (resp. $p$) is image of~$t$ (resp. $s$) in $A$ (resp. $R$) and residue field~$F_q$. Observe also that $R$ imbeds into $A=R\oplus R\pi$ and that $\m=\r\cap R$.

Note as well that $\sigma$ induces an involution, say $*$, on $A$
which fixes every element of $R$ and sends $\pi$ to $-\pi$. The fixed ring of $*$ is $R$. Setting
$$
S=\{a\in A\,|\, a^*=-a\},
$$
we have
$$
A=R\oplus S\text{ and }S=R\pi.
$$
Observe  that $a-a^*\in\r$ for all $a\in A$, so the involution that $*$ induces on $A/\r$ is trivial.

Let $h:V\times V\to A$ be a nondegenerate hermitian or skew hermitian form defined on a nonzero free $A$-module $V$ of finite rank $r$,
and let $\U_r(A)$ be the associated unitary group. We wish to embed $\U_r(A)$ into a symplectic group $\Sp(V,f)$, where $f:V\times V\to R$ is a nondegenerate alternating form arising from $h$, in order to study the restriction of the Weil representation of $\Sp(V,f)$ to $\U_r(A)$.

Suppose first $m=2\ell$ is even. Then $A$ is a free $R$-module with basis $1,\pi$. We take $h$ to be hermitian. Given $u,v\in V$, we have
$$
h(u,v)=k(u,v)+f(u,v)\pi
$$
for unique $k(u,v),f(u,v)\in R$, and \cite[\S 3]{HS} shows that $f$ is indeed nondegenerate. Herman and Szechtman \cite{HS} studied the Weil representation of $\U_r(A)$ by means of its imbedding into $\Sp_{2r}(R)$.
They obtained a decomposition of the Weil module into irreducible constituents as well as a description of each constituent
by means of Clifford theory with respect to the largest abelian congruence subgroup of $\U_r(A)$. Their study required
auxiliary material on hermitian forms, unitary groups and their actions on hermitian spaces, all of which was developed in \cite{CHQS}
for this purpose.

In this paper we consider the last remaining case, namely the case when $m=2\ell-1$ is odd. In this case,
neither $A$ nor $V$ are free $R$-modules, as seen \S\ref{riform}, except in the special case $\ell=1$ when $A=R$. We take $h$ to be skew hermitian. Given $u,v\in V$, we have
$$
h(u,v)=f(u,v)+k(u,v)\pi
$$
for a unique $f(u,v)\in R$ and some $k(u,v)\in R$. As shown in \S\ref{riform}, $f$ is in fact nondegenerate. Virtually none of
the tools from \cite{CHQS} are now available and this prompted the development of \cite{CS} in order to provide the subsidiary material
on skew hermitian geometry and their corresponding unitary groups required to study the Weil representation.
In particular, the rank of $V$ is now forced to be even, say $r=2n$, as shown in \cite[Proposition 2.12]{CS}.

The goal of this paper is to study the
Weil module of $\U_{2n}(A)$ by means of its imbedding
into $\Sp(V,f)$ and to obtain its decomposition into
irreducible submodules. It turns out that this Weil module is multiplicity free and, much
as in the symplectic case studied by Cliff, McNeilly and Szechtman
\cite{CMS1}, the Weil module has a top layer and a bottom layer,
the constituents of the top layer are the various eigenspaces for
the action of the center of $\U_{2n}(A)$, and the bottom layer is
a Weil module for a unitary group $\U_{2n}(\overline{A})$, where
$\overline{A}$ is a quotient of~$A$. The degrees of the irreducible Weil
components are determined. All of these constituents
remain irreducible and non-isomorphic to each other when
restricted to $\SU_{2n}(A)$ provided $n>1$. In spite
of the fact that $V$ is not a free $R$-module, we are able to obtain a
decomposition of the Weil module of $\Sp(V,f)$ into $2\ell$
non-isomorphic modules, a case which falls outside of the study
made in \cite{CMS1}.

It should be noted that while in \cite{S} and \cite{HS} the irreducible constituents of the top layer of the Weil
representation do not have the same degrees, the opposite is true here. Moreover, in \cite{HS} there are two choices of nondegenerate hermitian form, while only one nondegenerate skew hermitian form
is possible in this paper. One of the choices yields a Weil representation of $\U_2(A)$ whose top layer has fewer irreducible constituents
than the other choice. Nothing like this happens in this paper. Furthermore, in the present paper $\U_{2n}(A)$ is an extension of the symplectic group $\Sp_{2n}(q)$ by a maximal congruence subgroup, whereas in \cite{HS} (resp. \cite{GS})
the unitary group $\U_n(A)$ is an extension of
an orthogonal group $\mathrm{O}_n(q)$ (resp. the unitary group $\U_n(q)$). 

The paper is organized as follows. Background material from various sources is collected in \S\ref{riform} and \S\ref{hes}. Orbit enumeration required for the Weil decomposition is carried out in \S\ref{ors}, while the decomposition itself is presented in \S\ref{dec}. The degrees
of the irreducible constituents can be found in \S\ref{degf}. The final section gives a decomposition of the Weil representation of $\Sp(V,f)$,
where $V$ is not a free $R$-module.

All representations are assumed to be finite dimensional over the
complex numbers.

\section{Rings and forms}\label{riform}

We maintain throughout the above notation with $m=2\ell-1$ odd. Observe that the annihilator of $\pi$ in $R$ is~$Rp^{\ell-1}$. Now $\pi p^{\ell-1}=\pi^{2\ell-1}=0$ and $A=R\oplus R\pi$, so $Rp^{\ell-1}=A\pi^{2(\ell-1)}$ is the minimal ideal of both $R$ and $A$, and will be denoted by $\n$.

The nilpotency degree of $\r$ is $2\ell-1$ and that of $\m$ is $\ell$, so that $|A|=q^{2\ell-1}$ and $|R|=q^\ell$. In particular, either $\ell=1$ and $A=F_q=R$, or else $\ell>1$ and $A$ is not a free $R$-module.

An alternative construction of $(A,*)$ is to start off with $R$ as above, so that $R$ is a finite, commutative, local, principal ideal ring with maximal ideal $\m=R p$ having nilpotency degree~$\ell\geq 1$ and residue field $F_q$ of odd characteristic, and then set
$$A=R[X]/(X^2-p, X^{2\ell-1}).$$ Then $R$ imbeds into $A=R\oplus R\pi$, where $\pi$ is the image of $X$ in $A$,
the annihilator of $\pi$ in $R$ is $Rp^{\ell-1}$, and $\pi^2=p$. The involution of $R[X]$ that fixes every element of $R$
and sends $X$ to $-X$ induces an involution $*$ on $A$.

We have a group homomorphism, called the norm map, $A^\times\to R^\times$, given by $a\mapsto aa^*$. Its kernel will be denoted by $N$.
Thus
$$
N=\{a\in A^\times\,|\, aa^*=1\}.
$$
Arguing as in \cite[Lemmas 7.4 and 7.5]{HS}, we see that
\begin{equation}\label{norman}
|N|=2|A|/|R|.
\end{equation}

Recall that $V$ be a nonzero free $A$-module of finite rank $r$ and that $h:V\times V\to A$ is a nondegenerate
skew hermitian form. The latter means that $h$ is linear in the second variable, $*$-linear in the first, $h(v,u)=-h(u,v)^*$
for all $u,v\in V$, and the $A$-linear map $V\to V^*=\Hom_A(V,A)$, $u\mapsto h(u,-)$, is a monomorphism. Here $V^*$ is viewed as an $A$-module via $(a\alpha)(v)=a^*\alpha(v)$. Since $A$ has a unique minimal ideal, we deduce the existence of a linear character $A^+\to\C^*$ having no nonzero
ideals on its kernel. Then the first half of the proof of \cite[Lemma 2.1]{CMS2} applies (just replace $R$ by $A$ and $[~,~]$ by $h$)
to show that $V\to V^*$ is an isomorphism. It follows from \cite[Proposition 2.12]{CS} that $r=2n$ is even and $V$ has a ``symplectic" basis, i.e., a basis $u_1,\dots,u_n,v_1,\dots,v_n$ satisfying
\begin{equation}\label{symbas}
h(u_i,v_j)=\delta_{ij},\; h(u_i,u_j)=h(v_i,v_j)=0.
\end{equation}
In particular, the isomorphism type of the corresponding unitary group
$$
\U=\U(V,h)=\{g\in\GL_A(V)\,|\, h(gu,gv)=h(u,v)\text{ for all }u,v\in V\}
$$
does not depend on the choice of $h$.

Identifying each $a\in N$ with $a\cdot 1_V$, we can view $N$ as a central subgroup of $\U$.

Recalling the meaning of $f$, we see that $f:V\times V\to R$ is given by
$$
f(v,w)=\frac{h(v,w)+h(v,w)^*}{2},\quad v,w\in V.
$$
We readily verify that $f$ is an alternating $R$-bilinear form on $V$, which is now viewed as an $R$-module. We refer to $f$ as the alternating form associated to $h$.

We claim that $f$ is nondegenerate, in the sense that the associated $R$-linear map $V\to V^*=\Hom_R(V,R)$, $v\mapsto f(v,-)$,
is a monomorphism. Indeed, let $v$ be a nonzero element of $V$. Suppose, if possible, that $f(v,V)=0$. Then $h(v,V)\subseteq S=R\pi$. By
the nondegeneracy of $h$, there is $u$ in $V$ such that $h(v,u)\neq 0$. Since $h(v,u)\in R\pi$, we have $h(v,u)=r\pi^{2i-1}$ for some $r\in R^\times$ and $1\leq i<\ell$. Set $w=\pi u$. Then $h(v,w)=r\pi^{2i}=rp^i$, so $f(v,w)=rp^i$, which is nonzero because $r\in R^\times$ and $i<\ell$. This contradiction proves that $V\to V^*$ is injective. That $V\to V^*$ is also surjective is proven in \cite[Lemma 2.1]{CMS2}.

Now every $g\in\U$ is naturally an invertible $R$-linear map from $V$ to~$V$ preserving~$f$, so $\U$ is a subgroup of the symplectic group
$$
\Sp=\Sp(V,f)=\{g\in\GL_R(V)\,|\, f(gu,gv)=f(u,v)\text{ for all }u,v\in V\}.
$$
It is worth noting that if $\ell>1$ then $V$ is not a free $R$-module. Indeed, let $e_1,\dots,e_{2n}$ be an $A$-basis of $V$.
Then
$$
V=Re_1\oplus\cdots \oplus Re_{2n}\oplus  R\pi e_1\oplus\cdots \oplus R\pi e_{2n}.
$$
Suppose $V$ were $R$-free. Since $P=R\pi e_1\oplus\cdots \oplus R\pi e_{2n}$ is a direct summand of~$V$, then $P$ is a projective $R$-module. Since $R$ is local, it follows that $P$ is free. But $p^{\ell-1}P=0$. When $\ell>1$ neither $P$ nor $p^{\ell-1}$ are 0, so freeness
is contradicted.

Given an ideal $\i$ of $R$ we set
$$
V(\i)=\{v\in V\,|\, f(v,V)\subseteq \i\}.
$$
It is clear that $V(\i)$ is an $\Sp$-invariant submodule of $V$.

\begin{lemma}\label{orth3} We have
$$
\r^{2i-1}V=V(\m^i),\quad i\geq 1.
$$
In particular, $\r^j V$ is an $\Sp$-invariant $R$-submodule of $V$ for every $j\geq 0$.
\end{lemma}

\begin{proof} Let $v,w\in V$. Then $h(v,w)=r+s\pi$, where $r,s\in R$, so
$$h(v,\pi^{2i-1} w)=\pi^{2i-1}(r+s\pi)=sp^{i}+r\pi^{2i-1},$$
whence $f(v,\pi^{2i-1} w)=s p^{i}\in \m^{i}$. This proves that $\r^{2i-1}V \subseteq V(\m^{i})$.
Suppose conversely that $u\in V(\m^{i})$. By
\cite[Proposition 2.12]{CS}, the module $V$ has a basis $u_1,\dots,u_n,v_1,\dots,v_n$ satisfying (\ref{symbas}).
We have $$u=a_1u_1+\cdots+a_n u_n+b_1v_1+\cdots+b_n v_n, \quad a_k,b_k\in A.$$ Thus
$$
(b_k+b_k^*)/2=f(u_k,u)\in R p^{i},\quad 1\leq k\leq n.
$$
Moreover, since
$$
\pi (b_k-b_k^*)/2=f(-\pi u_k,u)\in R p^{i}=R\pi\pi^{2i-1},\quad 1\leq k\leq n,
$$
we see that
$$
(b_i-b_i^*)/2\in R\pi^{2i-1},\quad 1\leq k\leq n.
$$
It follows that
$$
b_k\in R p^{i}+R\pi^{2i-1}=\r^{2i-1},\quad 1\leq k\leq n.
$$
A similar argument using $v_k$ instead of $u_k$, shows that all $a_k\in \r^{2i-1}$, so that $u\in\r^{2i-1}V$.

The above shows all $\r^{2i-1} V$ are $\Sp$-invariant $R$-submodules of $V$. Given that $\r^{2i}V=\m^i V$
is also an $\Sp$-invariant $R$-submodule of $V$, the proof is complete.
\end{proof}

Given an $R$-submodule $U$ of $V$ we will write
$$
U^\perp=\{v\in V\,|\, f(v,U)=0\}.
$$
As shown in \cite[Lemma 2.1]{CMS2}, we have
\begin{equation}\label{size}
|V|=|U||U^\perp|.
\end{equation}
It is clear that $U^\perp$ is an $\Sp$-invariant submodule of $V$ provided $U$ is.
This is the case, for instance, when $U=I V$ and $I$ is an ideal of $R$. If $J$ denotes the
annihilator of $I$ in $R$, then \cite[Lemma 5.4]{CMS2} gives
\begin{equation}\label{orth}
(J V)^\perp =V(I)\text{ and }V(J)^\perp=I V.
\end{equation}

\begin{lemma}\label{orth4} Let $\i$ be an ideal of $A$ with annihilator $\j$. Then
$$
(\i V)^\perp=\j V.
$$
\end{lemma}

\begin{proof} Suppose first $\i=\r^{2i-1}$, where $1\leq i\leq\ell$. Then by Lemma \ref{orth3} and (\ref{orth}), we have
$$
(\r^{2i-1}V)^\perp=V(\m^i)^\perp=\m^{\ell-i}V=\r^{2(\ell-i)}V,
$$
where $\j=\r^{2(\ell-i)}$ is the annihilator of $\i$. Suppose next $\i=\r^{2i}$, where $0\leq i<\ell$. Then by Lemma \ref{orth3} and (\ref{orth}), we have
$$
(\r^{2i}V)^\perp=(\m^i V)^\perp=V(\m^{\ell-i})=\r^{2(\ell-i)-1}V,
$$
where $\j=\r^{2(\ell-i)-1}$ is the annihilator of $\i$.
\end{proof}

\section{Schr$\ddot{\rm o}$dinger and Weil representation}\label{hes}

A linear character $R^+\to\C^\times$ is said to be primitive if its kernel does not contain any ideal of $R$ but $(0)$.
Note that $R^+$ has $q^\ell-q^{\ell-1}>0$ primitive linear characters. We fix one of them, say $\lambda:R^+\to\C^\times$.

Let $H=R\times V$ be the Heisenberg group, with multiplication
$$
(r,u)(s,v)=(r+s+f(u,v),u+v).
$$
We identify the center $Z(H)=(R,0)$ of $H$ with $R^+$. Note that $\Sp$ acts on $H$ by means of automorphisms via
$$
{}^g (r,u)=(r,gu).
$$
Given an $R$-submodule $U$ of $V$, we consider the normal subgroup $H(U)=(R,U)$ of $H$.

\begin{prop}\label{schro} The Heisenberg group $H$ has a
unique irreducible module, up to isomorphism such that $Z(H)$ acts on it via $\lambda$.
Its dimension is equal to $\sqrt{|V|}$.
\end{prop}

\begin{proof} This can be found in \cite[Proposition 2.1]{CMS1}.
\end{proof}

We fix a Schr$\ddot{\rm o}$dinger representation $S:H\to\GL(X)$ of type $\lambda$, that is,
a representation satisfying the conditions stated in Proposition \ref{schro}. For $g\in\Sp$, the conjugate representation $S^g:H\to\GL(X)$, given by
$$
S^g(r,u)=S(r,gu),
$$
is also irreducible and lies over $\lambda$. By Proposition \ref{schro}, $S$ and $S^g$ are equivalent. By \cite[Theorem 3.1]{CMS1}
there is a representation $W:\Sp\to\GL(X)$ such that
\begin{equation}
\label{defew}
W(g)S(k)W(g)^{-1}=S({}^g k),\quad g\in\Sp, k\in H.
\end{equation}
We may thus view $X$ as a module for the semidirect product $H\rtimes \Sp$.

\begin{definition}\label{defw} Let $G$ be a subgroup of $\Sp$. By a Weil representation of $G$ of type~$\lambda$ we
understand any representation $W':G\to\GL(X)$ satisfying (\ref{defew}) for all $g\in G$ and $k\in H$. Since $S$ is irreducible,
Schur's Lemma ensures that the Weil representations of $G$ of type $\lambda$
are of the form
$$
g\mapsto \tau(g)W(g),\quad g\in G,
$$
where $\tau$ is a linear character of $G$.
\end{definition}

Let $U$ be a totally isotropic $R$-submodule $U$ of $V$ relative to $f$. Then $(0,U)$ is a subgroup of $H$ and
we define $X(U)$ to be the fixed points of $(0,U)$ in $X$, that is,
$$
X(U)=\{x\in X\,|\, S(0,u)x=x\text{ for all } u\in U\}.
$$
Note that if the submodule $U$ is $\Sp$-invariant, then the subgroup $(0,U)$ is normalized by $\Sp$, whence
$X(U)$ is a $\C \Sp$-submodule of $X$.

\begin{prop}\label{clifford} Let $U$ be a totally isotropic $R$-submodule of $V$ relative to $f$. Then $X(U)$
is an irreducible $H(U^\perp)$-submodule of $X$ of dimension $\sqrt{|U^\perp/U|}$.
\end{prop}

\begin{proof} This is shown in \cite[Proposition 4.1 and Lemma 4.2]{CMS1}.
\end{proof}

Given a finite group $G$, two $G$-modules $P$ and $Q$ and a finite $G$-set~$Y$, we set
$$
[P,Q]_G=\dim_\C\Hom_{\C G}(P,Q),
$$
$$
O_G(Y)=\text{number of }G\text{-orbits of }Y.
$$
By \cite[Theorem 4.5]{CMS1} for any subgroup $G$ of $\Sp(V)$, we have
\begin{equation}\label{number}
[X,X]_G=O_G(V).
\end{equation}

\section{Orbits}\label{ors}

\begin{lemma}\label{val} Recalling that $S=\{a\in A\,|\, a^*=-a\}$, we have
$$
\{h(u,u)\,|\, u\in V\setminus \r V\}=S.
$$
\end{lemma}

\begin{proof} Since $h$ is skew hermitian, $h(u,u)\in S$ for all $u\in V$.
By virtue of \cite[Proposition 2.12]{CS}, $V$ has an $A$-basis $u_1,\dots,u_n,v_1,\dots,v_n$
satisfying (\ref{symbas}). Then
$$
T=\{u_1+\pi r v_1\,|\, r\in R\}
$$
is a subset of $V\setminus \r V$ satisfying
$$
\{h(u,u)\,|\, u\in T\}=S.
$$
\end{proof}

\begin{lemma}\label{orbso} We have
$$
O_\Sp(V)=2\ell.
$$
\end{lemma}

\begin{proof} We claim that the $\Sp$-orbits of $V$ are
$$
V\setminus \r V,\; \r V\setminus \r^2 V,\;\r^2 V\setminus \r^3 V,\;\dots,\; \r^{2\ell-2}\setminus \{0\},\;\{0\}.
$$

First of all, $V\setminus \r V$ is an $\Sp$-orbit. Indeed, we know from \cite[Lemma 5.2]{CMS2} that $V\setminus V(\m)$ is
an $\Sp$-orbit while, on the other hand, Lemma \ref{orth3} ensures that $\r V=V(\m)$.

If $\ell=1$ we are done. Suppose henceforth that $\ell>1$. Care is required, as the actions of $\Sp$ and $A$ on $V$ do not commute.

We next verify that $\r V\setminus \r^2 V$ is an $\Sp$-orbit. Note that this set is $\Sp$-stable,
as $\r V\setminus \r^2 V=V(\m)\setminus \m V$. By \cite[Theorem 3.1]{CS}, if $u,v\in V\setminus \r V$,
then $u$ and $v$ are in the same $\U$-orbit if and only if $h(u,u)=h(v,v)$.
Let $u_1,\dots,u_n,v_1,\dots,v_n$ and $T$ have the same meaning as in the proof of the Lemma \ref{val}.
Since the actions of $\U$ and $A$ on $V$ commute and $\U$ is a subgroup of $\Sp$, it follows that every $w\in \r V\setminus \r^2 V$
is $\Sp$-conjugate to an element of $\pi T$. We wish to show that $\pi T$ is inside an $\Sp$-orbit.
To see this, let $\pi u_1+prv_1$ be an arbitrary element of $\pi T$. For $u,w\in V$ satisfying $f(u,w)=0$
the Eichler transformation $\tau_{u,w}$ of $V$, given~by
$$
v\mapsto v+f(v,u)w+f(v,w)u,\quad v\in V,
$$
is easily seen to belong to $\Sp$. Given $t\in R$, we set
$$
u=tv_1,\; w=\pi v_1,\; v=\pi u_1+prv_1.
$$
Note that
$$
f(v,u)=0,\; f(v,w)=-p, f(u,w)=0.
$$
Therefore $\tau_{u,w}\in\Sp$ and
$$
\tau_{u,w}(v)=\pi u_1+p rv_1-p t v_1=\pi u_1+p(r-t)v_1.
$$
By varying $t\in R$, we see that $\pi T$ is inside an $\Sp$-orbit. This proves that
$\r V\setminus \r^2 V$ is an $\Sp$-orbit.

Finally, since $$\r^2 V\setminus \r^3 V=p V\setminus p\r V=p(V\setminus \r V),$$
and $V\setminus \r V$ is an $\Sp$-orbit, it follows that $\r^2 V\setminus \r^3 V$ is an $\Sp$-orbit. Likewise,
since $$\r^3 V\setminus \r^4 V=p\r V\setminus p\r^2 V=p(\r V\setminus \r^2 V),$$
and $\r V\setminus \r^2 V$ is an $\Sp$-orbit, it is clear that $\r^3 V\setminus \r^4 V$ is an $\Sp$-orbit, and so on.
\end{proof}

\begin{lemma}\label{ols} We have
$$
O_\U(V\setminus \r V)=|S|.
$$
\end{lemma}

\begin{proof} If $\ell=1$ then any nonzero vector of $V$ is part of a symplectic basis and $|S|=1$, so the result follows.
Suppose next $\ell>1$. By \cite[Theorem 3.1]{CS}, if $u,v\in V\setminus \r V$ then $u,v$ are in the same $\U$-orbit
if and only if $h(u,u)=h(v,v)$. Now apply Lemma \ref{val}.
\end{proof}

\begin{lemma}\label{igual} We have
$$
O_\U(V\setminus \r V)=O_\U(\r V\setminus \r^2 V).
$$
\end{lemma}

\begin{proof} Let $E$ be a set of representatives for the
$\U$-orbits of $V\setminus \r V$. It is clear that every vector in
$\r V\setminus \r^2 V$ is $\U$-conjugate to a vector in $\pi E$. Thus the
map $E\to \pi E$, given by $e\mapsto \pi e$, is surjective. We claim
that it is also injective and, in fact, that if $e_1,e_2\in E$ and
$\pi e_1,\pi e_2$ are $\U$-conjugate then $e_1=e_2$. For this purpose we
recall that $\n=\m^{\ell-1}$, we view $\r V$ as a module for $A/\n $ and consider the map
$q:\r V\times \r V\to A/\n$ given by
$$
q(\pi v,\pi w)=h(v,w)+\n,\quad v,w\in V.
$$
Given $g\in \U$, for all $v,w\in V$ we have
$$
q(g\pi v,g\pi w)=q(\pi gv,\pi gw)=h(gv,gw)+\n =h(v,w)+\n=q(\pi v,\pi w),
$$
so the restriction of $g$ to $\pi V$ preserves $q$. Suppose $g\in \U$
satisfies $g\pi e_1=ye_2$. By above,
$$q(\pi e_2,\pi e_2)=q(g\pi e_1,g\pi e_1)=q(\pi e_1,\pi e_1),$$
which means
$$
h(e_1,e_1)-h(e_2,e_2)\in \n\cap S=(0).
$$
Thus $e_1,e_2$ are $\U$-conjugate by \cite[Theorem 3.1]{CS}. Since
$e_1,e_2\in E$, we infer $e_1=e_2$.
\end{proof}

We introduce new notation, provided $\ell>1$. Let $\overline{V}=V/\r^{2\ell-3}V$, which is a free module over $\overline{A}=A/\r^{2\ell-3}$ of rank $2n$.
For $a\in A$ and $v\in V$ set
$$
\overline{a}=a+\r^{2\ell-3}\in\overline{A}\text{ and }\overline{v}=v+\r^{2\ell-3}V\in\overline{V}.
$$
Note that $\overline{A}$ inherits an involution $\overline{*}$ from $A$, given by
$$
\overline{a}^{\overline{*}}=\overline{a^*},\quad a\in A.
$$
We further let
$\overline{R}=R/\m^{\ell-1}$, and for $r\in R$ we set
$$
\overline{r}=r+\m^{\ell-1}\in\overline{R}.
$$
We see that $\overline{R}$ imbeds into $\overline{A}$, that $\overline{R}$ is the fixed ring of $\overline{*}$ in $\overline{A}$,
and that $\overline{A}=\overline{R}\oplus \overline{R}\overline{\pi}$, where $\overline{\pi}^2=\overline{p}$ and the annihilator
of $\overline{\pi}$ in $\overline{R}$ is $\overline{R}\overline{p}^{\ell-2}$.

Observe that $h$ gives rise to a
nondegenerate skew hermitian form $\overline{h}:\overline{V}\times
\overline{V}\to \overline{A}$, given by
$$
\overline{h}(\overline{v},\overline{w})=\overline{h(v,w)}, \quad v,w\in V.
$$
Let $\overline{\U}$ stand for the associated unitary group.

\begin{lemma}\label{orbu} We have
$$
O_\U(V)=2(q^{\ell-1}+q^{\ell-2}+\cdots+1).
$$
\end{lemma}

\begin{proof} By induction on $\ell$. If $\ell=1$ then, by Lemma \ref{orbso}, $\U=\Sp$ has only two orbits on $V$, namely $V\setminus\{0\}$ and $\{0\}$. Suppose $\ell>1$ and the result is true for $1\leq \ell'<\ell$. By Lemmas \ref{ols} and \ref{igual},
we have
$$
O_\U(V\setminus \r^2 V)=2|S|=2q^{\ell-1}.
$$
By inductive hypothesis,
$$
O_{\overline\U}(\overline{V})=2(q^{\ell-2}+\cdots+q+1).
$$
On the other hand, $\r^2 V$ is also a free $\overline{A}$-module of rank $2n$ endowed with a nondegenerate skew hermitian form
$h':\r^2 V\times \r^2 V\to \overline{A}$, given by
$$
h'(\pi^2 v,\pi^2 v)=h(v,w)+\r^{2\ell-3},\quad v,w\in V.
$$
It is clear that the map $\Delta:\r^2 V\to \overline{V}$, given by
$$
\pi^2 v\mapsto \overline{v},\quad v\in V,
$$
is a well-defined isometry. Moreover, $\Delta$ commutes with the actions of $U$ on $\r^2 V$ (by restriction)
and $\overline{V}$ (inherited from the action of $\U$ on $V$). Furthermore, by \cite[Theorem 4.1]{CS}, the canonical map $\U\to\overline{\U}$ is an epimorphism. It follows that
$$
O_\U(\r^2 V)=O_\U(\overline{V})=O_{\overline\U}(\overline{V})=2(q^{\ell-2}+\cdots+q+1),
$$
so
$$
O_\U(V)=O_\U(V\setminus\r^2 V)+O_\U(\r^2 V)=2(q^{\ell-1}+q^{\ell-2}+\cdots+1).
$$
\end{proof}

\section{Decomposition of $X$ into irreducible $\U$-modules}\label{dec}

Suppose $\ell>1$ and keep the notation introduced prior to Lemma~\ref{orbu}. By Lemma \ref{orth3}, $\r^{2\ell-3}V=V(\m^{\ell-1})$, so that $\r^{2\ell-3}V$ is an $\Sp$-invariant $R$-submodule of $V$.
Note that $f$ gives rise to a
nondegenerate alternating form $\overline{f}:\overline{V}\times
\overline{V}\to \overline{R}$, given by
$$
\overline{f}(\overline{v},
\overline{w})=\overline{f(v,w)},\quad v,w\in V.
$$
Observe that $\overline{f}$ is the alternating form associated to $\overline{h}$. Let $\overline{\Sp}$ and $\overline{H}$  be the associated symplectic and Heisenberg groups.

We have natural group homomorphism $\U \to\overline{\U}$, say $g\mapsto \overline{g}$, shown to be surjective in \cite[Theorem 4.1]{CS},
and having kernel
$$
\Omega=\{g\in\U\,|\, gv\equiv v\mod \r^{2\ell-3}V\}.
$$

Recalling that $\n=\m^{\ell-1}=\r^{2\ell-2}$, we see that $\n^2=(0)$, so $\n V$ is a totally isotropic $R$-submodule of $V$ relative to $f$. Setting
$$
\Bot=X(\n V),
$$
we see that $\Bot$ is an $\Sp$-submodule of $X$. 

It is easy to see that $W(g)|_{Bot}$ must be a scalar operator for any $g\in\Omega$. Indeed, by Lemma \ref{orth4}, we have $(\n V)^\perp=\r V$.
Thus, by Proposition \ref{clifford}, $Bot=X(\n V)$ is an irreducible $H(\r V)$-module. But $R$ acts on $X$ by scaling through $\lambda$, so
$Bot$ is an irreducible $(\m,\r V)$-module. On the other hand, since the actions
of $\U$ and $A$ on $V$ commute, we readily verify that $\Omega$ acts trivially on every element of $(\m,\r V)/(0,\n V)$. We infer that
$$
W(g)|_{Bot}S(k)|_{\Bot} W(g)|_{Bot}^{-1}=S({}^g k)|_{\Bot}=S(k)|_{Bot}, \; g\in\Omega, k\in (\m,\r V).
$$
It follows from Schur's Lemma that $W(g)|_{Bot}$ is a scalar operator for every $g\in\Omega$, as claimed. For our purposes, we
need the stronger result that, for a suitable choice of $W$, $W(g)|_{Bot}=1_\Bot$ for every $g\in\Omega$.

To prove this we consider the map $(\m, \r V)\to \overline{H}$ given by
\begin{equation}
\label{yus} (r p,\pi v)\mapsto (-\overline{r},\overline{v}),\quad r\in R,v\in V.
\end{equation}
It is a well-defined group epimorphism with kernel
$(0,\n V)$. Since the actions of $\U$ and $A$ on $V$ commute, we see that (\ref{yus}) is compatible
with the actions $\U$ on $H$ and $\overline{H}$, in the sense that if $g\in\U$ and $\overline{k}\in\overline{H}$ is the element corresponding to $k\in (\m,\r V)$ then
\begin{equation}
\label{yus2} \overline{{}^g k}=\overline{{}^g}\, \overline{k}.
\end{equation}

Let $\overline{\lambda}$ be the primitive linear character of $\overline{R}$ given by
$$\overline{r}\mapsto \lambda(pr),\quad r\in R.$$
Let $\overline{S}:\overline{H}\to\GL(\Bot)$ be the representation of $\overline{H}$ obtained via
(\ref{yus}), that is,
\begin{equation}
\label{yus3}
\overline{S}(\overline{k})=S(k)|_{\Bot},\quad k\in (\m,\r V).
\end{equation}
It follows from Proposition \ref{clifford} that $\overline{S}$ is irreducible and it is clear from (\ref{yus3}) that
the center of $\overline{H}$ acts on $\Bot$ via $\overline{\lambda}$. Thus $\overline{S}$ is Schr$\ddot{\rm o}$dinger of type $\overline{\lambda}$ by Proposition \ref{schro}.

Let $\overline{W}:\overline{\Sp}\to\GL(\Bot)$ be a Weil representation of $\overline{\Sp}$ corresponding to $\overline{S}$.
The compatibility condition (\ref{yus2}) then gives
\begin{equation}
\label{yus4}
\overline{W}(\overline{g})\overline{S}(\overline{k})\overline{W}(\overline{g})^{-1}=
\overline{S}(\overline{{}^g}\, \overline{k})=\overline{S}(\overline{{}^g k}),\quad g\in\U, k\in (\m,\r V).
\end{equation}
Let $W_0:\U\to\GL(\Bot)$ be the representation defined by
$$
W_0(g)=\overline{W}(\overline{g}),\quad g\in\U.
$$
Then (\ref{yus3}) and (\ref{yus4}) give
$$
W_0(g)S(k)|_{\Bot} W_0(g)^{-1}=S({}^g k)|_{\Bot}, \quad g\in\U, k\in (\m,\r V).
$$
Since $\Bot$ is an irreducible $(\m,\r V)$-module via $k\mapsto S(k)|_{\Bot}$ as well as a $\U$-invariant submodule of $X$, it follows from Schur's Lemma that there is a linear character $\tau:\U\to\C$ such that
$$
\tau(g)W(g)|_{\Bot}=W_0(g),\quad g\in \U.
$$
According to Definition \ref{defw}, $g\mapsto \tau(g)W(g)$, $g\in\U$, is a Weil representation of type~$\lambda$. We have shown:

\begin{theorem}\label{botw} Suppose that $\ell>1$. Then there is a Weil representation of $\U$ of type $\lambda$ through which the
congruence subgroup $\Omega$ acts trivially on $\Bot$ and the corresponding representation of $\overline{\U}$ afforded by $\Bot$
is a Weil representation of type~$\overline{\lambda}$.
\end{theorem}

Recall that $N=\{a\in A^\times \,|\, aa^*=1\}$. For $\phi\in\Irr(N)$, let
$$
\varepsilon_\phi=\frac{1}{|N|}\underset{a\in N}\sum \phi(a^{-1})a
$$
be the idempotent in $\C N$ associated to $\phi$. Given a $\C N$-module $Y$, we set
$$
Y(\phi):=\varepsilon_\phi Y=\{y\in Y\,|\, ay=\phi(a)y\text{ for all }a\in N\}.
$$
We refer to $Y(\phi)$ as the $N$-eigenspace of $Y$ associated to $\phi$. We know that $Y$ is the direct
sum of its $N$-eigenspaces, that is,
$$
Y=\underset{\phi\in \Irr(N)}\bigoplus Y(\phi).
$$
In general, one or more $Y(\phi)$ may be zero.

We apply this to the $\Sp$-module
$$
Top=X/\Bot.
$$
Since $N$ is a central subgroup of $\U$, each $\Top(\phi)$ is a $\U$-submodule of~$X$.

\begin{theorem}\label{nonzero} Suppose that $\ell>1$. Then $\Top(\phi)\neq (0)$ for each $\phi\in\Irr(N)$.
\end{theorem}

\begin{proof} By \cite[Propositon 2.12]{CS}, $V$ has a basis $u_1,\dots,u_n,v_1,\dots,v_n$ satisfying (\ref{symbas}).
Let $M$ (resp. $Q$) be the $A$-span of $u_1,\dots,u_n$ (resp. $v_1,\dots,v_n$) so that $V=M\oplus Q$, where
each of $M,Q$ is a maximal totally isotropic $A$-submodule of $V$ relative to $h$. Since $|M|=\sqrt{|V|}$ and $f(M,M)=0$,
it follows from (\ref{size}) that $M$ a maximal totally isotropic $R$-submodule of $V$ relative to $f$.

Take $U=M$ to construct $X$ as in the proof of \cite[Proposition 2.1]{CMS1}. Then the $(0,v)$, $v\in Q$, form a transversal for $H(U)$ in $H(V)$,
so that
$$
X=\underset{v\in Q}\bigoplus (0,v)\C y.
$$
It follows that the $e_v=(0,v)y$, $v\in Q$, form a complex basis of $X$.  Let $T$ be a transversal for $\n Q$ in $\r Q$. For $t\in T$, consider the element $E_t$ of $X$ defined by
$$
E_t=\underset{v\in \n Q}\sum e_{t+v}.
$$
By construction, the $E_t$, $t\in T$, are linearly independent. Therefore, by Lemma~\ref{orth3}, (\ref{orth})
and Proposition \ref{clifford}, the span of all $E_t$ has dimension
$$
\begin{aligned}
|T|&=|\r Q/\n Q|=\sqrt{|\r V/\n V|}=\sqrt{|V(\m)/\n V|}\\
&=\sqrt{|(\n V)^\perp/\n V|}=\dim X(\n V)=\dim \Bot.
\end{aligned}
$$
We claim that, in fact, $Bot$ is equal to the span of all $E_t$. By above, it suffices to show that $(0,\n V)$ fixes every $E_t$.
To see this, let $w\in \n V$ and $v\in \r Q$. Then $w=w_M+w_Q$ for unique $w_M\in\n M$ and $w_Q\in\n Q$. Since $\n^2=(0)$, we have $f(w_M,w_Q)=0$. Moreover, because $\n\m=0$, we have
$$
f(w_M,v)\in\n (\r\cap R)=\n\m=(0).
$$
Therefore
$$
\begin{aligned}
(0,w)e_v &=(0,w_Q+w_M)e_v=(0,w_Q)(0,w_M)e_v=(0,w_Q)(0,w_M)(0,v)y\\
&=(0,w_Q)(0,v)(0,w_M)y=(0,w_Q)(0,v)y=(0,w_Q)e_v=e_{v+w_Q}.
\end{aligned}
$$
Applying this to any $t\in T$, we see that
$$
(0,w)E_t=E_t,
$$
as required.

Consider the representation $P:N\to\GL(X)$ given by
\begin{equation}
\label{a-1}
P(a)e_v=e_{av},\quad a\in N,v\in Q.
\end{equation}
We next claim that
\begin{equation}
\label{a0}
P(a)S(k)P(a)^{-1}=S({}^a k),\quad a\in N,k\in H(V).
\end{equation}
To verify (\ref{a0}) we resort to the following formulae, which give the action of $H(V)$ on basis vectors $e_v$, $v\in Q$:
\begin{equation}
\label{a1}
S(0,w)e_v=e_{v+w},\quad w\in Q,
\end{equation}
\begin{equation}
\label{a2}
S(0,u)e_v=\lambda(2f(u,v))e_{v},\quad u\in M,
\end{equation}
\begin{equation}
\label{a3}
S(r,0)e_v=\lambda(r)e_{v},\quad r\in R.
\end{equation}
It is easy to use (\ref{a-1}) and (\ref{a1})-(\ref{a3}) to verify (\ref{a0}). All we require is that $N$ preserves both $Q$ and $M$ (since they are $A$-submodules of $V$) and that $N$ is a subgroup of $\U$ and hence of $\Sp$.

Since $S$ is irreducible, Schur's Lemma ensures the existence of a linear character $\psi$ of $N$ such that
$$
W(a)=P(a)\psi(a),\quad a\in N.
$$
Since multiplication by $\psi$ merely shuffles the $N$-eigenspaces of $\Top$, it suffices to prove that these eigenspaces are nonzero
when $N$ acts on $X$ via $P$. For this purpose, let $Y$ be the subspace of $X$ spanned by all $e_v$, $v\in Q\setminus \r Q$.
It is clear that $Y$ is $N$-stable. Moreover, as $Bot$ was shown to be the span of all $E_t$, we see that $\Bot\cap Y=(0)$. Thus, $Y$ imbeds as an $N$-submodule of the $N$-module $\Top$.
Let $\phi\in\Irr(N)$. It suffices to show that $Y(\phi)\neq (0)$. Let $v\in Q\setminus\r Q$. By (\ref{a-1}), we have
\begin{equation}
\label{a4}
\varepsilon_\phi e_v=\underset{a\in N}\sum \phi(a^{-1})e_{av}.
\end{equation}
Since $v\in Q\setminus \r Q$, we have $cv=0\Leftrightarrow c=0$ for all $c\in A$, whence $av=bv\Leftrightarrow a=b$ for all $a,b\in N$.
As the $e_w$, $w\in Q\setminus \r Q$, are linearly independent, it follows that (\ref{a4}) is a nonzero
element of $Y(\phi)$.
\end{proof}

\begin{theorem}\label{m2} Any Weil module $X$ has $2(q^{\ell-1}+q^{\ell-2}+\cdots+1)$ irreducible constituents for $\U$, all non isomorphic to each other.
\end{theorem}

\begin{proof} By induction on $\ell$. In the classical case $\ell=1$ it is well-known (see Theorem \ref{algo} below)
that $X$ has two non isomorphic irreducible constituents for $\U=\Sp$.

Suppose $\ell>1$ and the result is true for $1\leq \ell'<\ell$. By Theorem \ref{nonzero}, $\Top$ is the direct sum of
its $N$-eigenspaces, all of which are nonzero. By (\ref{norman}), there are
$$
|N|=2|A|/|R|=2q^{\ell-1}
$$
such summands. By Theorem \ref{botw}, we may assume without loss of generality
that $W$ is chosen so that the congruence subgroup $\Omega$  acts trivially on $\Bot$ and
the corresponding representation of $\U/\Omega \cong \overline{\U}$ is a Weil representation of primitive type.
By inductive hypothesis, $\Bot$ is the direct sum of $2(q^{\ell-2}+\cdots+q+1)$ of nonzero $\U$-submodules.
On the other hand, by (\ref{number}) and Lemma \ref{orbu}, we have
$$
[X,X]_\U=O_\U(V)=2(q^{\ell-1}+q^{\ell-2}+\cdots+1).
$$
It follows that the above summands of $X$ are irreducible $\U$-modules non isomorphic to each other.
\end{proof}

\begin{note} Suppose $\ell>1$. The above proof shows that the irreducible constituents of $\Top$ are the $|N|=2q^{\ell-1}$ eigenspaces for the action of $N$ and, up to multiplication by a linear character of $\U$, the remaining $2(q^{\ell-2}+\cdots+q+1)$ irreducible constituents of $X$ are the irreducible constituents of the Weil module $\Bot$ for the unitary group $\overline{\U}$, inflated to~$\U$.
\end{note}

Let $\SU$ stand for the special unitary group, namely the subgroup of~$\U$ consisting of all unitary transformations whose
determinant is equal to~1.

\begin{theorem} Suppose $n>1$. Then any Weil module $X$ has exactly $2(q^{\ell-1}+q^{\ell-2}+\cdots+1)$ irreducible constituents for $\SU$, all non isomorphic to each other. Thus, all $\U$-constituents of $X$ remain $\SU$-irreducible and non isomorphic to each other.
\end{theorem}

\begin{proof} By induction on $\ell=1$. Suppose first $\ell=1$. Then $\U=\Sp$. Since $\Sp$ is generated symplectic transvections, every symplectic transformation must have determinant equal to~1, whence $\Sp=\SU$.  We have already noted that $X$ has two non isomorphic irreducible constituents for $\Sp$, so the result is true in this case.

Suppose $\ell>1$ and the result is true for $1\leq \ell'<\ell$. Arguing as in the proof of Theorem \ref{m2}, we are reduced to showing
\begin{equation}
\label{ful}
O_\SU(V)=2(q^{\ell-1}+q^{\ell-2}+\cdots+1).
\end{equation}

In order to verify (\ref{ful}), we claim first of all that if $u,v\in V\setminus \r V$ satisfy $h(u,u)=h(v,v)$
then $u,v$ are in the same $\SU$-orbit. Indeed, by \cite[Theorem 3.1]{CS} there is $g\in\U$
such that $gu=v$. Let $a=\det g$. We readily see that $a\in N$. Since $v\in V\setminus \r V$, there is $w\in V\setminus \r V$
such that $h(v,w)=1$. Then $V=W\perp W^\perp$, where $W$ is free of rank 2 with basis $\{v,w\}$ and $W^\perp$ is free of rank $2(n-1)$.
Since $n>1$, there is $t\in\U$ such that $t|_W=1_W$ and $t|_{W^\perp}$ is multiplication by $a^{-1}$. Then $tg\in\SU$ sends $u$ to~$v$.

Having shown that \cite[Theorem 3.1]{CS} is valid for $\SU$, it follows that Lemmas \ref{ols} and \ref{igual} are also valid for $\SU$,
that is,
$$
O_\SU(V\setminus \r V)=|S|=O_\SU(\r V\setminus \r^2 V).
$$

Next observe that the proof of Lemma \ref{orbu} hinges on Lemmas \ref{ols} and~\ref{igual} as well as on the fact that the canonical homomorphism $\U\to \overline{\U}$ is surjective. We have already noted that Lemmas \ref{ols} and \ref{igual} are true for $\SU$ and we know from \cite[Theorem 10.2]{CS2} that the canonical homomorphism $\SU\to \overline{\SU}$ is also surjective. It follows that Lemma \ref{orbu} is valid for $\SU$, i.e., (\ref{ful}) is correct.
\end{proof}

\section{Dimensions of the irreducible constituents of the $\U$-module $X$}\label{degf}

Here we find the dimensions of all irreducible components of the Weil module $X$ for $\U$.
We keep throughout the notation introduced in the proof of Theorem \ref{nonzero} and let $P:\Sp\to\GL(X)$ be a projective
representation satisfying
$$
P(g)S(k)P(g)^{-1}=S({}^g k),\quad g\in \Sp,k\in H(V),
$$
which extends the choice made in (\ref{a-1}). Let $X_+$ and $X_-$ be the eigenspaces of $P(-1_V)$ with eigenvalues $1$ and $-1$,
respectively. As shown in \cite[\S 3]{CMS1}, the subspaces $X_+$ and $X_-$ are nonzero and invariant under all $P(g)$, $g\in\Sp$.
We may thus consider the function $c:\Sp\to\C^\times$, given by
\begin{equation}
\label{cg}
c(g)=(\det P(g)|_{X_+})^{-1}(\det P(g)|_{X_-}),\quad g\in \Sp.
\end{equation}
According to \cite[Theorem 3.1]{CMS1}, the map
$W:\Sp\to\GL(X)$, given by
$$
W(g)=c(g)P(g),\quad g\in \Sp,
$$
is a Weil representation. We are particularly interested in $c(a)$, $a\in N$. Recalling the decomposition
$$
N=\{1,-1\}\times (1+\r)\cap N,
$$
we have the following values for $c(a)$, $a\in N$.

\begin{lemma} We have $c(a)=1$ if $a\in (1+\r)\cap N$ and
$$
c(-1)=(-1)^{(|Q|-1)/2}=\begin{cases} 1 & \text{ if }q^n\equiv 1\mod 4,\\
-1 & \text{ if }q^n\equiv -1\mod 4.\end{cases}
$$
\end{lemma}

\begin{proof} Since $P$ restricts to a group homomorphism on $N$, so does $c$. To determine the linear character $c|_N:N\to\C^\times$,
let $L$ be a subset of $Q\setminus\{0\}$ obtained by selecting exactly one element out of every subset $\{v,-v\}$ of $Q\setminus\{0\}$.
Then $e_0$ and the $e_v+e_{-v}$, $v\in L$, form a basis of $X_+$ and the $e_v-e_{-v}$, $v\in L$, form a basis of $X_-$. It follows
from (\ref{a-1}) and (\ref{cg}) that the image of $c|_N$ is contained in $\{1,-1\}\subset\C^\times$. Since $(1+\r)\cap N$ has odd
order, we deduce that $c(a)=1$ if $a\in (1+\r)\cap N$. Moreover, it is clear from (\ref{a-1}) and (\ref{cg}) that $c(-1)=(-1)^{(|Q|-1)/2}$,
as required.
\end{proof}

Let $G$ be the group of all linear characters of $(1+\r)\cap N$ and consider the following subgroups of $G$:
$$
G_i=\{\phi\in G\,|\, N\cap (1+\r^{2\ell-1-i})\subseteq \ker\phi\},\quad 0\leq i\leq 2\ell-2.
$$
Then
$$
G=G_0\supseteq G_1\supseteq\cdots\supseteq G_{2\ell-2}=\{1\}.
$$
Let $X^\pm$ denote the eigenspaces of $W(-1_V)$ with eigenvalues $\pm 1$. For $\phi\in G$,
let $X^\pm(\phi)=\{x\in X^\pm\,|\, ax=\phi(a)x\text{ for all }a\in (1+\r)\cap N\}$.

\begin{prop}\label{dimx} Let $\phi\in G$. Then $\dim X^\pm(\phi)$ is equal to
$$
\frac{q^{(2\ell-1)n}-q^{(2\ell-3)n}}{2q^{\ell-1}},\quad \phi\in G_0\setminus G_2,
$$
$$
\frac{q^{(2\ell-1)n}-q^{(2\ell-3)n}}{2q^{\ell-1}}+\frac{q^{(2\ell-3)n}-q^{(2\ell-5)n}}{2q^{\ell-2}},\quad \phi\in G_2\setminus G_4,
$$
$$
\frac{q^{(2\ell-1)n}-q^{(2\ell-3)n}}{2q^{\ell-1}}+\frac{q^{(2\ell-3)n}-q^{(2\ell-5)n}}{2q^{\ell-2}}+
\frac{q^{(2\ell-5)n}-q^{(2\ell-7)n}}{2q^{\ell-3}},\quad \phi\in G_4\setminus G_6,
$$
$$
\vdots
$$
$$
\frac{q^{(2\ell-1)n}-q^{(2\ell-3)n}}{2q^{\ell-1}}+\frac{q^{(2\ell-3)n}-q^{(2\ell-5)n}}{2q^{\ell-2}}+\cdots+
\frac{q^{3n}-q^{n}}{2q}, \quad \phi\in G_{2\ell-4}\setminus G_{2\ell-2},
$$
$$
\frac{q^{(2\ell-1)n}-q^{(2\ell-3)n}}{2q^{\ell-1}}+\cdots+
\frac{q^{3n}-q^{n}}{2 q}+\begin{cases}\frac{q^n\pm 1}{2} & \text{ if }q^n\equiv 1\mod 4,\\
\frac{q^n\mp 1}{2} & \text{ if }q^n\equiv -1\mod 4\end{cases},\quad \phi\in G_{2\ell-2}=\{1\}.
$$
\end{prop}

\begin{proof} By definition,
$$
|G_i|=|N\cap (1+\r)|/|N\cap (1+\r^{2\ell-1-i})|.
$$
We claim that
$$
G_0=G_1\supset G_2=G_3\supset\cdots\subset G_{2\ell-4}=G_{2\ell-3}\supset G_{2\ell-2}=\{1\},
$$
with sizes
$$
q^{\ell-1}=q^{\ell-1}>q^{\ell-2}=q^{\ell-2}>\cdots>q=q>1.
$$
Indeed, as shown in \cite[\S 7]{HS}, for every proper ideal $\i$ of $A$ we have
$$
\frac{|N\cap (1+\r)|}{|N\cap (1+\i)|}=\frac{|A|}{|R|}\frac{|R\cap\i|}{|\i|}=q^{\ell-1}\frac{|R\cap\i|}{|\i|}.
$$
Now
$$
\frac{|\r|}{|R\cap\r|}=q^{\ell-1}, \frac{|\r^2|}{|R\cap\r^2|}=q^{\ell-2}, \frac{|\r^3|}{|R\cap\r^3|}=q^{\ell-2},
$$
$$
\frac{|\r^4|}{|R\cap\r^4|}=q^{\ell-3}, \frac{|\r^5|}{|R\cap\r^5|}=q^{\ell-3},\dots.
$$
We see that, for $j\in\{1,\dots,2\ell-1\}$, we have
$$
\frac{|N\cap (1+\r)|}{|N\cap (1+\r^j)|}=\begin{cases} q^{(j-1)/2} & \text{ if }j\text{ is odd}\\
 q^{j/2} & \text{ if }j\text{ is even}.
\end{cases}
$$
Setting $j=2\ell-1-i$, for $i\in\{0,\dots,2\ell-2\}$, we obtain
$$
|G_i|=\frac{|N\cap (1+\r)|}{|N\cap (1+\r^{2\ell-1-i})|}=\begin{cases} q^{2(\ell-1)-i/2} & \text{ if }i\text{ is even},\\
 q^{2\ell-1-i/2} & \text{ if }i\text{ is odd}.
\end{cases}
$$
This explains the above sizes and equalities between the $G_i$.

The action of $N\cap (1+\r)$ on $X$ is the same as on the permutation module $\C Q$.  Now $\C Q$ has the following $N\cap (1+\r)$-invariant subspaces:
\begin{equation}
\label{insu}
\C[Q\setminus \r Q], \C[\r Q\setminus \r^2 Q],\dots, \C[\r^{2\ell-2} Q\setminus \{0\}], \C \{0\}.
\end{equation}

The key observation is that $\phi$ appears in $\C[Q \setminus \r Q]$ if and only if $\phi \in G_0$, $\C[\r Q\setminus \r^2 Q]$ if and only if $\phi\in G_1$, in $\C[\r^2 Q\setminus \r^3 Q]$ if and only if $\phi\in G_2$, and so on.   At the end, $\phi$ appears in
$\C[\r^{2\ell-2} Q\setminus \r^{2\ell-1} Q]$ if and only if $\phi\in G_{2\ell-2}=\{1\}$.  This is because the orbits in the action of $N \cap (1 + \r)$ on $\C [\r^i Q \setminus \r^{i+1} Q]$ are regular orbits for the quotient $(N \cap (1 + \r))/(N \cap (1 + \r^{2\ell-1-i})$ for $i = 0, \dots, 2\ell - 1$.   In particular, every $\phi \in G_i$ occurs with the same multiplicity in the action of $N \cap (1 + \r)$ on $\C [\r^i Q \setminus \r^{i+1} Q]$.

Since all linear characters $\phi \in G_i$ that occur in one of these $N\cap (1+\r)$-invariant subspaces occur with the same multiplicity, the dimension of the $\phi$-part of each $N\cap (1+\r)$-invariant subspace is equal to the number of orbits of $N \cap (1+\r)$.
The number of $N \cap (1 + \r)$-orbits in each subspace appearing in (\ref{insu}) is respectively equal to
$$
\begin{array}{c}
(|Q|-|\r Q|)/|G_0|, (|\r Q|-|\r^2 Q|)/|G_1|, (|\r^2 Q|-|\r^3 Q|)/|G_2|,\dots, \\
(|\r^{2\ell-2} Q|-|\r^{2\ell-1} Q|)/|G_{2\ell-2}|, 1.
\end{array}
$$
The quantities $|\r^i Q\setminus \r^{i+1} Q|$ are easy to compute, and we already computed $|G_i|$, which yields a formula for the dimension of $X(\phi)$ for each $\phi\in G$.  Any $\phi \in G$ can be extended in two ways to a linear character of $N$.  When $\phi \ne 1$, each $X(\phi)$ will be split exactly in half in this process.  In the special case of the trivial character $\phi=1$, the dimensions of $X^{\pm}(\phi)$ agree with the $\ell = 1$ case of Theorem \ref{algo}.   This gives the indicated formulae.
\end{proof}

\begin{theorem} The dimensions of the irreducible constituents of $X$ under $\U$ are
\begin{equation}
\label{dimtop}
\frac{q^{(2\ell-1)n}-q^{(2\ell-3)n}}{2q^{\ell-1}}, \frac{q^{(2\ell-3)n}-q^{(2\ell-5)n}}{2q^{\ell-2}},\dots, \frac{q^{3n}-q^{n}}{2q}, \frac{q^n\pm 1}{2}.
\end{equation}
More precisely, if $\ell=1$ these dimensions are $\frac{q^n\pm 1}{2}$, while if $\ell>1$ then the dimensions of all irreducible constituents of $\Top$ under $\U$ are all equal to
$$
\frac{\dim\Top}{|N|}=\frac{q^{(2\ell-1)n}-q^{(2\ell-3)n}}{2q^{\ell-1}},
$$
while the remaining dimensions stated in (\ref{dimtop}) correspond to the irreducible constituents of $\Bot$ under $\U$.
 \end{theorem}

\begin{proof} By induction on $\ell$. The case $\ell=1$ is covered by Theorem \ref{algo}. Suppose $\ell>1$
and the result is true for $\ell-1$. Setting $\overline{A}=A/\r^{2\ell-3}$, we have the groups
$\overline{N}, \overline{G}, \overline{G}_i$, $0\leq i\leq 2\ell-4$, corresponding to $\overline{A}$.

Note that the action of $N\cap (1+\r)$ on the basis $E_{t+\n V}$, $t\in \r Q$, of $\Bot$ is like the above action but for the quotient space $Q/\r^{2\ell-3} Q\cong \r Q/\n Q$, that is, the action of $N\cap (1+\r)$ on $\Bot$ is the same as on the permutation module $\C[Q/\r^{2\ell-3} Q]$. This gives $\dim\Bot^\pm(\overline{\phi})$ for every $\overline{\phi}\in \overline{G}$ by simply replacing $\ell$ by $\ell-1$ in the formulae of Proposition \ref{dimx}.

Now a $\phi\in G_0\setminus G_2$ does not give rise to any linear character of $\overline{N}\cap (1+\overline{\r})$,
so that $\phi$ does not enter $\Bot$, and after that $G_2\setminus G_4$ corresponds with $\overline{G_0}\setminus \overline{G_2}$,
$G_4\setminus G_6$ corresponds with $\overline{G_2}\setminus \overline{G_4}$, and so on. The resulting cancellations in
$\dim \Top^\pm(\phi)=\dim X^\pm(\phi)-\dim\Bot^\pm(\phi)$ give stated dimensions
of all irreducible constituents $\Top^\pm(\phi)$ of $\Top$. The others follow by inductive hypothesis via Theorem~\ref{botw}.
\end{proof}

\section{Decomposition of $X$ into irreducible $\Sp$-modules}

The action of the central element $-1_V$ of $\Sp$ on $X$ is determined in \cite[\S 3]{CMS1}. If $\ell>1$ then $V$ is not a free $R$-module, and it is conceivable that in this case a linear character $\Sp\to\C^*$ be nontrivial on $-1_V$, thereby altering the dimensions of the eigenspaces of $-1_V$ acting on $X$. This is not the case. Indeed, let $u_1,\dots,u_n,v_1,\dots,v_n$ be an $A$-basis of $V$ satisfying (\ref{symbas}) and set
$U_1=Ru_1\oplus\cdots\oplus Ru_n\oplus R v_1\cdots\oplus Rv_n$, $U_2=R\pi u_1\oplus\cdots\oplus R\pi u_n\oplus R\pi v_1\cdots\oplus R\pi v_n$.
Then $V=U_1\perp U_2$. Moreover, the symplectic groups associated to $U_1$ and $U_2$ are isomorphic to $\Sp_{2n}(R)$ and $\Sp_{2n}(R/\n)$,
respectively. These groups are perfect, except when $(n,q)=(2,3)$, but even in this case the nontrivial central element is in
the derived subgroup, as the index of the derived subgroup in $\SL_2(R)$ and $\SL_2(R/\n)$ is 3. All in all, $-1_{U_1}$ and $-1_{U_2}$
are in the derived subgroups of $\Sp(U_1)$ and $\Sp(U_2)$, respectively, so $-1_V$ is in the derived subgroup of $\Sp$.

\begin{theorem}\label{algo} The Weil module $X$ has $2\ell$ irreducible constituents for $\Sp$, all non isomorphic to each other,
with dimensions as indicated below.
\end{theorem}

\begin{proof} The classical case $\ell=1$ follows from (\ref{number}),
the fact that $\U=\Sp$ has two orbits on $V$, and that the eigenspaces $X^\pm$ of $-1_V$ acting on $X$ are nontrivial.
As for their dimensions, we have (see \cite[\S 3]{CMS1}, for instance)
$$
\dim X^\pm=\begin{cases}\frac{q^n\pm 1}{2} & \text{ if }q^n\equiv 1\mod 4,\\
\frac{q^n\mp 1}{2} & \text{ if }q^n\equiv -1\mod 4.\end{cases}
$$

Suppose $\ell>1$ and consider the ideals of $A$ square (0), namely the $\ell$ ideals
$$
0=\r^{2\ell-1}\subset \r^{2\ell-1}\subset\cdots\subset \r^{\ell}.
$$
It follows from Lemma \ref{orth3} that each $\r^i V$, $\ell\leq i\leq 2\ell-1$, is an $\Sp$-invariant
$R$-submodule of $V$, clearly totally isotropic with respect to $f$. By \cite[Corollary 4.3]{CMS1}
we have the following chain of $\ell+1$ \emph{distinct} $\Sp$-submodules of $X$:
$$
(0)\subset X(\r^\ell V)\subset\cdots\subset X(\r^{2\ell-2} V)\subset X(\r^{2\ell-1}V)=X,
$$
which gives rise to the $\ell$ factor modules
$$
Y_i=X(\r^{2\ell-i}V)/X(\r^{2\ell-(i+1)}V),\quad 1\leq i<\ell,
$$
and
$$
Y_\ell=X(\r^\ell V).
$$
Let $Y_i^\pm$ indicate the $\pm 1$-eigenspaces of $-1_V$ acting on $Y_i$. It follows from Lemma \ref{orth4} and \cite[Lemma 4.4]{CMS1}
that for $1\leq i<\ell$ we have
$$
\dim Y_i^\pm=\frac{\sqrt{|\r^{i-1} V/\r^{2\ell-i} V|}-\sqrt{|\r^{i} V/\r^{2\ell-(i+1)} V|}}{2}=\frac{q^{[2(\ell-i)-1]n}(q^{2n}-1)}{2}.
$$
Moreover, it follows from Lemma \ref{orth4} and the proof of \cite[Lemma 4.4]{CMS1} that
$$
\dim Y_\ell^\pm=\begin{cases}\frac{q^n\pm 1}{2} & \text{ if }q^n\equiv 1\mod 4,\\
\frac{q^n\mp 1}{2} & \text{ if }q^n\equiv -1\mod 4.\end{cases}
$$

In particular, the eigenspaces of $-1_V$ acting on each $Y_i$ are nontrivial.
Therefore, $X$ is the direct sum of $2\ell$ nonzero $\Sp$-submodules. We infer from (\ref{number})
that these modules are all irreducible and non isomorphic to each other.
\end{proof}


\end{document}